\documentclass{amsart}

\usepackage{amsmath,amscd,amssymb}

\usepackage{hyperref}
\hypersetup{colorlinks,linkcolor={red},citecolor={blue},urlcolor={blue}}

\newtheorem{theorem}{Theorem}[section]
\newtheorem{lemma}[theorem]{Lemma}
\newtheorem{proposition}[theorem]{Proposition}
\newtheorem{corollary}[theorem]{Corollary}
\newtheorem{conjecture}[theorem]{Conjecture}

\theoremstyle{definition}
\newtheorem{definition}[theorem]{Definition}
\newtheorem{example}[theorem]{Example}
\newtheorem{remark}[theorem]{Remark}

\numberwithin{equation}{section}

\newcommand{\CC}{\mathbb C}
\newcommand{\HH}{\mathbb H}

\newcommand{\NN}{\mathbb N}
\newcommand{\PP}{\mathbb P}
\newcommand{\QQ}{\mathbb Q}
\newcommand{\RR}{\mathbb R}
\newcommand{\ZZ}{\mathbb Z}

\newcommand{\cD}{\mathcal D}
\newcommand{\cH}{\mathcal H}

\newcommand{\SL}{\mathop{\mathrm {SL}}\nolimits}

\newcommand{\Sp}{\mathop{\mathrm {Sp}}\nolimits}
\newcommand{\Orth}{\mathop{\null\mathrm {O}}\nolimits}

\newcommand{\rank}{\mathop{\mathrm {rank}}\nolimits}

\newcommand{\latt}[1]{{\langle{#1}\rangle}}
\newcommand{\Kthree}{\mathop{\mathrm {K3}}\nolimits}

\newcommand{\Grit}{\operatorname{Grit}}
\newcommand{\Borch}{\operatorname{Borch}}
\def\div{\operatorname{Div}}
\newcommand{\m}{\operatorname{mod}}
\newcommand{\im}{\operatorname{Im}}
\newcommand{\Hum}{\operatorname{Hum}}
\newcommand{\mult}{\operatorname{mult}}
\newcommand{\norm}{\operatorname{Norm}}
\newcommand{\sing}{\operatorname{Sing}}

\begin{document}
\title[Theta block conjecture for paramodular forms of weight 2]
{Theta block conjecture for paramodular forms of weight 2}

\author{Valery Gritsenko}

\address{Laboratoire Paul Painlev\'{e}, Universit\'{e} de Lille, 59655 Villeneuve d'Ascq Cedex, France and National Research University Higher School of Economics, Russian Federation}

\email{Valery.Gritsenko@univ-lille.fr}

\author{Haowu Wang}

\address{Laboratoire Paul Painlev\'{e}, Universit\'{e} de Lille, 59655 Villeneuve d'Ascq Cedex, France}

\email{haowu.wangmath@gmail.com}


\date{\today}


\begin{abstract}
In this paper we construct an infinite family of paramodular forms of 
weight $2$ which are simultaneously Borcherds products and additive Jacobi 
lifts. This proves an important part of the theta-block conjecture of 
Gritsenko--Poor--Yuen (2013) related to the most important infinite series of 
theta-blocks of weight $2$ and $q$-order $1$. We also consider some 
applications of this result.
\end{abstract}

\maketitle

\section{Introduction}
Paramodular forms are Siegel modular forms of degree two 
with respect to the symplectic group $\Gamma_t$ of elementary divisor 
$(1,t)$, the paramodular group. 
There are two ways to construct paramodular forms from Jacobi forms. The 
first one is the additive Jacobi lifting due to Gritsenko (see 
\cite{G94} and \cite{G94a}) which lifts a holomorphic Jacobi form to a 
paramodular form. 
The second method is the multiplicative lifting (Borcherds automorphic 
product, see \cite{Bo95}, \cite{Bo98}) in a form, proposed by 
Gritsenko--Nikulin in \cite{GN98}, 
which sends a weakly holomorphic Jacobi form of weight $0$ to a meromorphic 
paramodular  form. 
In \cite{GPY1}, V. Gritsenko, C. Poor and D. Yuen 
investigated the paramodular forms which are simultaneously Borcherds 
products and Gritsenko lifts. 

Let $f:\NN \to \ZZ$ be a function with a finite support, where $\NN$ is the set of nonnegative integers.
We define {\it theta block} (see \cite{GSZ}) 
\begin{equation}\label{thblock}
\Theta_f(\tau,z)=\eta^{f(0)}(\tau)
\prod_{a=1}^{\infty}(\vartheta_a(\tau,z)/\eta(\tau))^{f(a)}
\end{equation}
as a finite  product of the  Jacobi theta-series 
$\vartheta_a(\tau,z)=\vartheta(\tau,az)$ divided by the Dedekind 
$\eta$-function 
$\eta(\tau)=q^{\frac{1}{24}}\prod_{n\ge 1}(1-q^n)$, where 
$$
\vartheta(\tau,z)= q^{\frac{1}{8}}
(\zeta^{\frac{1}{2}}-\zeta^{-\frac{1}{2}}) \prod_{n\geq 1} (1-q^{n}\zeta)
(1-q^n \zeta^{-1})(1-q^n)
$$
is the odd Jacobi theta-series which is a holomorphic Jacobi form of weight $\frac{1}{2}$ and index $\frac{1}{2}$ with a multiplier system of order $8$ (see \cite{GN98}). We call $\Theta_f$ a pure theta block if $f$ is nonnegative on $\NN$.
In general, the pure theta block $\Theta_f$ is a weak Jacobi form of weight
$f(0)/2$ and index $N=\frac{1}2\sum a^2f(a)$ with a character. 
It is important that 
for some functions $f$ the theta block is holomorphic at infinity and is a holomorphic Jacobi form (see the general theory of theta blocks in \cite{GSZ}).  
In this way, one gets the most significant holomorphic Jacobi forms of small 
weights. The following conjecture was proposed in \cite{GPY1}.
\smallskip

\noindent
{\bf Theta block conjecture.} 
{\it Let the pure theta block $\Theta_f$ be a holomorphic Jacobi form of weight 
$k$ and index $N$ with vanishing order one in $q=e^{2\pi i\tau}$. 
We define a weak Jacobi form $\Psi_f=-(\Theta_f\lvert T_{-}(2))/\Theta_f$ of weight $0$ and index $N$, where $T_{-}(2)$ is the index 
raising Hecke operator. Then
$$
\Grit(\Theta_f)=\Borch(\Psi_f)
$$
is a holomorphic symmetric paramodular form of weight $k$ with respect to $\Gamma_N^+$.}
\smallskip

This conjecture gives a characterization of paramodular forms which are 
simultaneously Borcherds products and Gritsenko lifts. The conjecture was 
proved in \cite[\S 8]{GPY1} for all known series of theta 
blocks of weights $3\leq k \leq 11$. Another known infinite series 
is given by the theta blocks of weight 2 of type 
$\frac{10-\vartheta}{6-\eta}$ found in \cite{GSZ}:
$$
\phi_{2,\mathbf{a}}=
\frac{\vartheta_{a_1}\vartheta_{a_2}\vartheta_{a_3}\vartheta_{a_4}
\vartheta_{a_1+a_2}\vartheta_{a_2+a_3}\vartheta_{a_3+a_4}
\vartheta_{a_1+a_2+a_3}\vartheta_{a_2+a_3+a_4}\vartheta_{a_1+a_2+a_3+a_4}}
{\eta^{6}}\in J_{2,N(\mathbf{a})}
$$
where $\mathbf{a}=(a_1, a_2, a_3, a_4)\in \ZZ^4$ and 
$$
N(\mathbf{a})=2a_1^2 + 3a_1a_2+2a_1a_3+a_1a_4+3a_2^2+4a_2a_3+2a_2a_4
+3a_3^2+3a_3a_4+2a_4^2.
$$
We remark that there are four infinite series of theta blocks of weight $2$ in \cite{GSZ}. The series $\phi_{2,\mathbf{a}}$ is the most important because it gives the first examples of holomorphic Jacobi forms of weight 2: the Jacobi-Eisenstein series of index $25$ and the Jacobi cusp form of index $37$ (see \S \ref{Subsec:4.2}).
\begin{theorem}\label{thmwt2}
For any  $\mathbf{a}\in \ZZ^4$ such that $\phi_{2,\mathbf{a}}$ is not 
identically zero, one has
$$\Grit(\phi_{2,\mathbf{a}})= 
\Borch\left(-\frac{\phi_{2,\mathbf{a}}|T_{-}(2)}{\phi_{2,\mathbf{a}}}
\right)\in M_2(\Gamma_{N(\mathbf{a})}^+).
$$
\end{theorem}

This theorem was announced in \cite{GW} with the  main idea of the proof. 
In this paper we give its complete proof and present some applications.

\section{Modular forms and lift constructions}

Let $N$ be a positive integer. The \textit{paramodular group} of level
(or polarization)  $N$ is a subgroup of $\Sp_2(\QQ)$ defined by
\begin{equation}
\Gamma_N=\left(\begin{array}{cccc}
  * & N* & * & * \\ 
  * & * & * & */N \\ 
  * & N* & * & * \\ 
  N* & N* & N* & *
  \end{array}   \right) \cap \Sp_2(\QQ),\quad \text{\rm all}\ *\in \ZZ.
\end{equation}
For $N>1$, we shall use the following double normal extension
\begin{equation}
\Gamma_N^+=\Gamma_N\cup \Gamma_NV_N,\quad V_N=\frac{1}{\sqrt{N}}\left(\begin{array}{cccc}
0 & N & 0 & 0 \\ 
-1 & 0 & 0 & 0 \\ 
0 & 0 & 0 & 1 \\ 
0 & 0 & -N & 0
\end{array}  \right).
\end{equation}
This group acts on the Siegel upper half plane of genus $2$ 
$$\HH_2=\left\lbrace Z=\left( \begin{array}{cc}
  \tau & z \\ 
  z & \omega
  \end{array}  \right)\in M(2,\CC): \im Z>0    \right\rbrace $$
in the usual way
$ M\latt{Z}=(AZ+B)(CZ+D)^{-1}$, $M=\left(\begin{array}{cc}
A & B \\ 
C & D
\end{array}  \right) \in \Sp_2(\RR)$. 
We put 
$(F\lvert_k M )(Z) =\det(CZ+D)^{-k}F(M\latt{Z})$
for any integer $k$.

\begin{definition}\label{JF1}
A holomorphic function $F: \HH_2 \longrightarrow \CC$ is called a
\textit{Siegel paramodular form} of  weight $k$ and level $N$ with 
character $\chi$ if $F\lvert_k M=\chi(M)F$ for any $M\in \Gamma_N$.
We denote the space of such modular forms by $M_k(\Gamma_N,\chi)$. A 
paramodular form $F$ is called a \textit{cusp} form if 
$\Phi(F\lvert_k g)=0$ for all $g\in \Sp_2(\QQ)$, 
where $\Phi$ is Siegel's operator.
The space of paramodular cusp forms is denoted by $S_k(\Gamma_N,\chi)$.
\end{definition}

Let $\chi_N: \Gamma_N^+\to \{\pm 1\}$ be the nontrivial binary character 
such  that $\chi_N(V_N)=-1$ and $\chi_N|_{\Gamma_N}=1$. Then  
$M_k(\Gamma_N)$ is decomposed into the direct 
sum of plus and minus $V_N$-eigenspaces.
For $F\in M_{k}(\Gamma_N^+, \chi_N^\epsilon)$ ($\epsilon=0$ or $1$),
we consider its Fourier and Fourier--Jacobi 
expansions
\begin{equation}\label{FJ}
F(Z)
=\sum_{m\geq 0} \sum_{\substack{n\in \NN, r\in \ZZ\\ 4nmN-r^2\geq 0}}
c(n,r,m)q^n\zeta^r \xi^{mN}
=\sum_{m\geq 0}\phi_{mN}(\tau,z)\xi^{mN},
\end{equation}
where 
$q=\exp(2\pi i \tau)$, $\zeta=\exp(2\pi i z)$, $\xi=\exp(2\pi i \omega)$.
Then we have the equality $(-1)^{k+\epsilon}F(\tau,z,\omega)
=F(N\omega,z,\tau/N)$, which yields
$c(n,r,m)=(-1)^{k+\epsilon}c(m,r,n)$.
When $k+\epsilon$ is odd (even), $F$ is called \textit{antisymmetric}
(\textit{symmetric}). We remark that $F$ is a cusp form if and only if 
$c(n,r,m)=0$ unless $4nmN-r^2>0$.

The Fourier--Jacobi coefficient $\phi_{mN}\in J_{k,mN}$
is a holomorphic Jacobi form (see \cite{EZ} and the more general 
Definition \ref{JF2} below). 

We introduce the additive lifting which is completely determined by 
its first Fourier--Jacobi coefficient.
Let $\phi(\tau,z)\in J_{k,N}$. 
The index 
raising Hecke operator $T_{-}(m)$ is defined as follows
$$ 
\phi\lvert_k T_{-}(m)=m^{-1}\sum_{\substack{ ad=m\\ b\m d }}a^{k}\phi
\left(\frac{a\tau+b}{d},az\right) \in J_{k,mN}.
$$
The Fourier coefficients of $\phi\lvert_k T_{-}(m)$ are given by
$$ 
c(n,r;\phi\lvert_k T_{-}(m))
=\sum_{\substack{d\in \NN\\ d\vert (n,r,m)}}
d^{k-1}c\left(\frac{nm}{d^2},\frac{r}{d};\phi \right).
$$
\begin{theorem}[\cite{G94}]
For $\phi\in J_{k,N}$, we have
$$ 
\Grit(\phi)(Z)=c(0,0;\phi)G_k(\tau)+ \sum_{m\geq 1} 
\left( \phi\lvert_k T_{-}(m) \right) (\tau,z)e^{2\pi i mN\omega}
\in M_k(\Gamma_N^+,\chi_N^k)
$$
where
$G_k(\tau)=(2\pi i)^{-k}(k-1)!\zeta(k)+\sum_{n\geq 1}\sigma_{k-1}(n)q^n$ is 
the Eisenstein series of weight $k$ on $\SL_2(\ZZ)$.
Moreover, if $\phi$ is a Jacobi cusp form then $\Grit(\phi)$ is a 
paramodular cusp form.
\end{theorem}

The paramodular form $\Grit(\phi)$ is always symmetric. 
We now describe the second  lifting, the Borcherds automorphic 
product (see \cite{Bo95} and \cite{Bo98}) in a form proposed by 
Gritsenko and Nikulin (see \cite{GN98}).
The Borcherds automorphic product is determined by the first two 
Fourier--Jacobi coefficients. 

\begin{theorem}[\cite{GN98}]\label{bpthm}
Let $N$ be a positive integer. Assume that $\Psi\in J_{0,N}^{!}$ is a 
weakly holomorphic Jacobi form of weight $0$ and index $N$
with Fourier expansion
$$\Psi(\tau,z)=\sum_{n,r\in\ZZ}c(n,r)q^n\zeta^r$$
and $c(n,r)\in \ZZ$ for $4Nn-r^2\leq 0$. We set
$C=\frac{1}{4}\sum_{r\in \ZZ}r^2 c(0,r)$.
Then the function 
\begin{equation}\label{bprod1} 
\Borch(\Psi)=\biggl(\eta^{c(0,0)}\prod_{r >0}\left(\frac{\vartheta_{r}}
{\eta} \right)^{c(0,r)}\xi^C\biggr) \cdot \exp\left(-\Grit(\Psi) \right)
\end{equation}
is a meromorphic paramodular form 
$\Borch(\Psi)\in M_k^{\rm{mero}}(\Gamma_N^+,\chi_\Psi)$ of weight $k=c(0,0)/2$
whose divisor in 
$\Gamma_N^+\backslash \HH_2$ consists of 
\textit{Humbert modular surfaces} 
\begin{equation*}
\Hum(T_0)
=\Gamma_N^+\latt{\{Z\in \HH_2: n_0\tau+r_0z+Nm_0\omega=0 \}}, 
\quad T_0=
\left(\begin{array}{cc}
n_0 & r_0/2 \\ 
r_0/2 & Nm_0
\end{array}  \right)
\end{equation*}
with $\gcd(n_0,r_0,m_0)=1$, $m_0\geq 0$ and $\det(T_0)<0$. 
The multiplicity of $\Borch(\Psi)$ on $\Hum(T_0)$ is $\sum_{n\geq 1}
c(n^2n_0m_0,nr_0)$. 
For $\lambda\gg 0$, on $\{Z\in \HH_2: \im Z > \lambda I_2 \}$ 
the following product expansion is valid 
$$
\Borch(\Psi)(Z)=q^A\zeta^B \xi^C\prod_{\substack{n,r,m\in\ZZ: m\geq 0 \\ 
\textit{if $m=0$ then $n\geq 0$}\\ \textit{if $m=n=0$ then $r<0$}  }}
(1-q^n\zeta^r\xi^{Nm})^{c(nm,r)}
$$ where 
\begin{align*}
24A=\sum_{r\in \ZZ}c(0,r),\quad
2B=\sum_{r\in \NN}r c(0,r),\quad D_0=\sum_{n<0}\sigma_0(-n) c(n,0).
\end{align*}
The character $\chi_\Psi$ of the paramodular form is generated by 
the character (or the multiplier system) of the theta block in 
$\Borch(\Psi)$ and the character $\chi_N^{k+D_0}$.
\end{theorem}

The automorphic product $\Borch(\Psi)$
is uniquely determined by the first two Fourier--Jacobi coefficients.
The first factor in \eqref{bprod1}
is a theta block $\Theta$ defined by the $q^0$-part of the Fourier expansion of $\Psi$.
The second one is the product $-\Theta \Psi$.
If $\Grit(\phi)$ is a Borcherds product, then 
$\phi=\Theta$ and $\phi\lvert T_{-}(2)=-\Theta\Psi$.
Thus, we conclude that $\phi$ is a theta block and 
$\Psi=-(\phi\lvert T_{-}(2))/\phi$.
Furthermore, we can show that $\phi$ has exactly vanishing order one in 
$q=e^{2\pi i\tau}$, otherwise the constant term $c(0,0)$ in the Fourier 
expansion of $-(\phi\lvert T_{-}(2))/\phi$ will be greater than $2k$ (see \cite[Proposition 7.2]{PSY} for a proof). We conjecture that $\phi$ should be a pure theta block, i.e. $c(0,r)\geq 0$ in $\Psi$. There is no example contrary to the claim at present.

In \cite{GPY1} the conjecture was proved
for the quasi-products of theta-functions  
$$
\eta^{3(8-l)}\vartheta_{d_1}\cdot \cdots \cdot \vartheta_{d_{l}} 
\in J_{12-l,N}, \qquad N=(d_1^2+\cdots d_\ell^2)/2\in \NN,
$$
and for the products of three {\it theta-quarks} 

$$
\theta_{a_1,b_1}\theta_{a_2,b_2}\theta_{a_3,b_3}
\in J_{3,N},\qquad
N=\sum_{i=1}^3(a_i^2+a_ib_i+b_i^2),
$$
where 
$\theta_{a,b}= {\vartheta_a\vartheta_b\vartheta_{a+b}}/{\eta}$
is a holomorphic Jacobi form of weight $1$ with
a character of order $3$ (see \cite{CG} and \cite{GSZ}).

The main idea of the proof of the conjecture for the mentioned 
above theta blocks is the following.
The odd theta-function $\vartheta(\tau, z)$ vanishes with order $1$
for $z=\lambda \tau+\mu$ for any $\lambda, \mu \in \ZZ$. Therefore 
we know the divisor of the theta block $\Theta_f$ of $q$-order one.
The Hecke operator $T_{-}(m)$ keeps this divisor of the theta block.  
Therefore we know a part of the divisor of the  lift of $\Theta_f$.
According to Theorem \ref{bpthm}, the divisor of $\Borch(\Psi)$ is 
determined by the singular Fourier coefficients $c(n,r)$ with 
$4nN-r^2<0$ of $\Psi$, where $N$ is the index of the Jacobi form 
$\Theta_f$. By the construction, $\Grit(\Theta_f)$ vanishes on the 
divisors determined by the $q^0$-term of $\Psi$. If $\Borch(\Psi)$ does not 
have another divisor then $\Grit(\Theta_f)/\Borch(\Psi)$ is a holomorphic 
modular form of weight zero and then equals a constant by K\"ocher's principle.
Unfortunately for this method of proof, $\Borch(\Psi)$ usually has additional divisors determined by 
singular Fourier coefficients from the higher $q^n$-terms.
In order to pass through  this difficulty, we represent $\Theta_f$ 
as a pull-back of a Jacobi form $\Theta_L$ in many variables associated to 
a certain positive definite lattice $L$. Since Jacobi forms in many 
variables have stronger 
symmetry, the function $\Psi_L=-\frac{\Theta_L \lvert T_{-}(2)}{\Theta_L}$ 
may have much simpler singular Fourier coefficients such that the divisor 
of $\Borch(\Psi_L)$ is determined only by the $q^0$-term of $\Psi_L$. 
In the next section, we give the corresponding definitions and results
in the case of many variables.

\section{Orthogonal modular forms and Jacobi forms of lattice index}

We start with the general setup (see \cite{Bo95}, \cite{CG} or \cite{GN18} for more details).  Let $M$ be an even lattice of signature 
$(2,n)$ with $n\geq 3$. Let
\begin{equation}
\cD(M)=\{[\omega] \in  \PP(M\otimes \CC):  (\omega, \omega)=0, (\omega,
\bar{\omega}) > 0\}^{+}
\end{equation}
be the corresponding  Hermitian symmetric domain of type IV (here $+$ 
denotes one of its two connected components). 
By $\Orth^+ (M)$ we denote the index $2$ subgroup of the integral 
orthogonal group  $\Orth(M)$ preserving $\cD(M)$.  
By $\widetilde{\Orth}^+(M)$ we denote the subgroup of $\Orth^+ (M)$ 
acting trivially on the discriminant group of $M$. For any $v\in M\otimes \QQ$ satisfying $
(v,v)<0$, the rational quadratic divisor associated to $v$ is defined as
\begin{equation*}
 \cD_v=\{ [Z]\in \cD(M) : (Z,v)=0\}. 
\end{equation*}
We assume that $M$ contains two hyperbolic planes and write 
$M=U\oplus U_1\oplus L(-1)$,
where $U=\ZZ e\oplus\ZZ f$ ($(e,e)=(f,f)=0$, $(e,f)=1$),
$U_1=\ZZ e_1\oplus\ZZ f_1$ are two hyperbolic planes 
and $L$ is an even integral positive definite lattice. 
We choose a basis of $M$ of the form $
(e,e_1,...,f_1,f)$, 
where $...$ denotes a basis of $L(-1)$.  
We fix a tube realization of the homogeneous domain $\cD(M)$ 
related to the 1-dimensional boundary component determined by the isotropic 
plane $F=\latt{e,e_1}$:
$$
\cH(L)=\{Z=(\tau,\mathfrak{z},\omega)\in \HH\times (L\otimes\CC)\times \HH: 
(\im Z,\im Z)>0\}, 
$$
where $(\im Z,\im Z)=2\im \tau \im \omega - 
(\im \mathfrak{z},\im \mathfrak{z})_L$. 
In this setting, a Jacobi form is a modular form with 
respect to the Jacobi group $\Gamma^J(L)$ which is the parabolic 
subgroup of $\Orth^+ (M)$ preserving the  isotropic plane $F$ and acting trivially on $L$.
This group is the semidirect product of $\SL_2(\ZZ)$ with the 
Heisenberg group $H(L)$ of $L$ (see \cite{G94} and \cite{CG}). Let $L^\vee$ denote the dual lattice of $L$ and $\rank(L)$ denote the rank of $L$.
We define Jacobi modular forms with respect to $L$.

\begin{definition}\label{JF2}
For $k\in\ZZ$, $t\in \NN$, a holomorphic 
function $\varphi : \HH \times (L \otimes \CC) \rightarrow \CC$ is 
called a {\it weakly holomorphic} Jacobi form of weight $k$ and index $t$ associated to $L$,
if it satisfies the functional 
equations
\begin{align}
\varphi \left( \frac{a\tau +b}{c\tau + d},\frac{\mathfrak{z}}{c\tau + d} 
\right)& = (c\tau + d)^k 
\exp{\left(i \pi t \frac{c(\mathfrak{z},\mathfrak{z})}{c 
\tau + d}\right)} \varphi ( \tau, \mathfrak{z} ),\\
\varphi (\tau, \mathfrak{z}+ x \tau + y)&= 
\exp{\bigl(-i \pi t ( (x,x)\tau +2(x,\mathfrak{z}))\bigr)} 
\varphi (\tau, \mathfrak{z} ),
\end{align}
for 
$\left(\begin{smallmatrix}
a & b \\ 
c & d
\end{smallmatrix}\right)   \in \SL_2(\ZZ)$, $x,y \in L$
and if it has a Fourier expansion  
\begin{equation}
\varphi ( \tau, \mathfrak{z} )= \sum_{n\geq n_0 }\sum_{\ell\in L^\vee}f(n,
\ell)q^n\zeta^\ell,
\end{equation}
where $n_0\in \ZZ$, $q=e^{2\pi i \tau}$ and $\zeta^\ell=e^{2\pi i (\ell,
\mathfrak{z})}$. 
If the Fourier expansion of $\varphi$ satisfies the condition
$(f(n,\ell) \neq 0 \Longrightarrow n \geq 0 )$
then $\varphi$ is called a {\it weak} Jacobi form.  If $( f(n,\ell) \neq 0 
\Longrightarrow 2n - (\ell,\ell) \geq 0 )$ (respectively $>0$)
then $\varphi$ is called a {\it holomorphic}
(respectively, {\it cusp}) Jacobi form. 
\end{definition}

We denote by $J^{!}_{k,L,t}$ (respectively, $J^{w}_{k,L,t}$, $J_{k,L,t}$, 
$J_{k,L,t}^{\text{cusp}}$) the vector space of weakly holomorphic Jacobi 
forms (respectively, weak, holomorphic or cusp Jacobi forms) 
of weight $k$ and index $t$. 
We note that the Jacobi forms  in one variable $J_{k,N}$
considered in the previous section are identical to the Jacobi forms 
$J_{k,A_1, N}$ for the lattice  $A_1=\latt{2}$ of rank $1$.
\smallskip

The additive Jacobi lifting and Borcherds product of Jacobi forms in many 
variables are  very similar to the case of one variable.
Let $\varphi \in J_{k,L,t}^{!}$.  For any positive integer $m$, we have 
\begin{equation}\label{T(m)}
\varphi \lvert_{k,t}T_{-}(m)(\tau,
\mathfrak{z})=m^{-1}\sum_{\substack{ad=m,a>0\\ 0\leq b <d}}a^k \varphi
\left(\frac{a\tau+b}{d},a\mathfrak{z}\right) \in J_{k,L,mt}^{!},
\end{equation}
and the Fourier coefficients of 
$\varphi \lvert_{k,t}T_{-}(m)(\tau,\mathfrak{z})$ are given by the formula
$$
f_m(n,\ell)=\sum_{\substack{a\in \NN\\ a \mid (n,\ell,m)}}a^{k-1} f
\left( \frac{nm}{a^2},\frac{\ell}{a}\right),
$$
where $a\mid(n,\ell,m)$ means that $a\mid (n,m)$ 
and $a^{-1}\ell\in L^\vee$. 

\begin{theorem}[see \cite{G94}]
Let $\varphi \in J_{k,L,1}$. Then the function 
$$ \Grit(\varphi)(Z)=f(0,0)G_k(\tau)+\sum_{m\geq 1}\varphi \lvert_{k,1}
T_{-}(m)(\tau,\mathfrak{z}) e^{2\pi i m\omega}
$$
is a modular form of weight $k$ for the stable orthogonal 
group $\widetilde{\Orth}^+(2U\oplus L(-1))$. Moreover, this modular form is symmetric i.e.
$\Grit(\varphi)(\tau,\mathfrak{z}, \omega)=
\Grit(\varphi)(\omega,\mathfrak{z}, \tau)$.

\end{theorem}

We fix an ordering $\ell >0$ in  the dual lattice $L^\vee$ in a way similar to  positive root systems (see the bottom of page 825 in \cite{G18}). The long paper \cite{G18} contains, in particular, 
the results of the preprints arXiv:1005.3753 and arXiv:1203.6503
of the first author. In the last preprint the following theorem was proved. 
\begin{theorem}[see Theorem 4.2 in \cite{G18} for details]
\label{th:Borcherds}
Let 
$$
\varphi(\tau,\mathfrak{z})=\sum_{n\in\ZZ, \ell\in L^\vee}f(n,\ell)q^n 
\zeta^\ell \in J^{!}_{0,L,1}.
$$
Assume that $f(n,\ell)\in \ZZ$ for all $2n-(\ell,\ell)\leq 0$. 
There is a meromorphic modular form of weight 
$f(0,0)/2$ and character $\chi$ with respect to  
$\widetilde{\Orth}^+(2U\oplus L(-1))$ 
$$ 
\Borch(\varphi)=
\biggl(\Theta_{f(0,*)}
(\tau,\mathfrak{z})\exp{(2\pi i\, C\omega)}\biggr)
\exp \left(-\Grit(\varphi)\right),
$$
where
$C=\frac{1}{2\rank(L)}\sum_{\ell\in L^\vee}f(0,\ell)(\ell,\ell)$ and
\begin{equation}\label{FJtheta}
\Theta_{f(0,*)}(\tau,\mathfrak{z})
=\eta(\tau)^{f(0,0)}\prod_{\ell >0}
\biggl(\frac{\vartheta(\tau,(\ell,\mathfrak{z}))}{\eta(\tau)} 
\biggr)^{f(0,\ell)}
\end{equation}
is a general theta block.
The character $\chi$ is induced by the character of the theta-product
and by the relation $\chi(V)=(-1)^D$, where
$V: (\tau,\mathfrak{z}, \omega) \to (\omega,\mathfrak{z},\tau)$, 
and $D=\sum_{n<0}\sigma_0(-n) f(n,0)$.

The poles and zeros of $\Borch(\varphi)$ lie on the rational quadratic 
divisors $\cD_v$, where $v\in 2U\oplus L^\vee(-1)$ is a primitive vector 
with $(v,v)<0$. The multiplicity of this divisor is given by 
$$ \mult \cD_v = \sum_{d\in \ZZ,d>0 } f(d^2n,d\ell),$$
where $n\in\ZZ$, $\ell\in L^\vee$ such that 
$(v,v)=2n-(\ell,\ell)$ and $v-(0,0,\ell,0,0)\in 2U\oplus L(-1)$.

The same function has the following infinite product expansion
$$
\Borch(\varphi)(Z)=q^A \zeta^{\vec{B}} \xi^C\prod_{\substack{n,m\in\ZZ, \ell
\in L^\vee\\ (n,\ell,m)>0}}(1-q^n \zeta^\ell \xi^m)^{f(nm,\ell)}, 
$$ 
where $Z= (\tau,\mathfrak{z}, \omega) \in \cH (L)$, 
$q=\exp(2\pi i \tau)$, 
$\zeta^\ell=\exp(2\pi i (\ell, \mathfrak{z}))$, $\xi=\exp(2\pi i \omega)$,
the notation $(n,\ell,m)>0$ means that either $m>0$, or $m=0$ 
and $n>0$, or $m=n=0$ and $\ell<0$, and
$$
A=\frac{1}{24}\sum_{\ell\in L^\vee}f(0,\ell),\quad
\vec{B}=\frac{1}{2}\sum_{\ell>0} f(0,\ell)\ell.
$$
\end{theorem}

\begin{remark}\label{rem:divisor}
In the case of arbitrary signature we can write the divisors of 
$\Borch(\varphi)$ in a way similar to the case of signature $(2,3)$ 
considered in Theorem \ref{bpthm}.
By the Eichler criterion (see \cite[proof of Theorem 3.1]{G94a} and 
\cite[Proposition 3.3]{GHS09}), if $v_1, v_2\in 2U\oplus L^\vee(-1)$ are 
primitive, have the same norm, and the same image in the discriminant 
group, i.e. $v_1-v_2\in 2U\oplus L(-1)$, then there exists $g\in 
\widetilde{\Orth}^+(2U\oplus L(-1))$ such that $g(v_1)=v_2$.
Therefore, for a primitive vector $v\in 2U\oplus L^\vee(-1)$ with $
(v,v)<0$, there exists a vector $(0,n,\ell,1,0)\in  2U
\oplus L^\vee(-1)$ such that $(v,v)=2n-(\ell,\ell)$,
$v- (0,n,\ell,1,0)\in 2U\oplus L(-1)$ and  
$$
\widetilde{\Orth}^+(2U\oplus L(-1))\cdot 
\cD_v=\widetilde{\Orth}^+(2U\oplus L(-1))\cdot\cD_{(0,n,\ell,1,0)}.
$$
\end{remark}

It is known that the Fourier coefficients of $\varphi\in J_{k, L, 1}^{!}$ satisfy (see \cite[Lemma 2.1]{G94a})
\begin{equation}\label{eq:property}
f(n,\ell)=f(n+(x,x)/2+(\ell,x),\ell + x), \; \text{for any $x\in L$}.
\end{equation}
Thus the Fourier coefficient $f(n,\ell)$ depends only 
on the (hyperbolic) norm $2n-(\ell,\ell)$ of its index
and the image of $\ell$ in the discriminant group of $L$.
The divisor of the Borcherds product in Theorem \ref{th:Borcherds}
is defined by the so-called  {\it singular} Fourier coefficients 
$f(n,\ell)$ with  $2n-(\ell,\ell)<0$.
There are only a finite number of orbits of such coefficients that are 
supported because the norm  $2n-(\ell,\ell)$ of  the indices of the nontrivial Fourier coefficients is restricted from below. 

We say that the  lattice $L$ satisfies the condition $\norm_2$
(see \cite{GH} and \cite{GN18}) if 
\begin{equation}\label{norm2}
\norm_2:\ \forall\, \bar{c} \in L^\vee/L \quad \exists\, h_c \in \bar{c} 
\quad \text{such that} \quad (h_c,h_c)\leq 2.
\end{equation} 

\begin{lemma}\label{LemmaN2} 
Let $\varphi\in J_{0, L, 1}^{w}$ and assume that $L$  satisfies the 
condition $\norm_2$. Then the singular Fourier coefficients 
of $\varphi$ and the divisors of $\Borch(\varphi)$
are determined entirely by the $q^0$-term of $\varphi$, i.e. any divisor 
of $\Borch(\varphi)$ is of the form $\cD_{(0,0,\ell_1,1,0)}$.
\end{lemma}
\begin{proof}
Suppose that $f(n,\ell)\neq 0$ is singular, i.e. $2n-(\ell, \ell)<0$.
There exists a vector $\ell_1 \in L^\vee$ 
such that $\ell - \ell_1\in L$ and $(\ell_1,\ell_1)\leq 2$ because $L$ 
satisfies $\norm_2$ condition. 
It is clear that $(\ell,\ell)-(\ell_1,\ell_1)$ is an even integer.
If $-2\leq 2n-(\ell, \ell)< 0$, it follows that 
$2n-(\ell, \ell)=-(\ell_1,\ell_1)$ and then $f(n,\ell)=f(0,\ell_1)$ 
by \eqref{eq:property}. If $2n-(\ell, \ell)<-2$, then 
there exists a negative integer $n_1$ satisfying 
$2n-(\ell, \ell)=2n_1-(\ell_1,\ell_1)$. Thus, there will be a nonzero 
Fourier  coefficient $f(n_1,\ell_1)$ with $n_1<0$, 
which contradicts the definition of weak Jacobi forms. 
According to Remark \ref{rem:divisor}, we have
$\widetilde{\Orth}^+(2U\oplus L(-1))\cdot 
\cD_{(0,n,\ell,1,0)}=\widetilde{\Orth}^+(2U\oplus L(-1))\cdot\cD_{(0,0,\ell_1,1,0)}$.
\end{proof}

Many examples of the lattices of type $\norm_2$ were found in \cite{GH}
and \cite{GN18} by the construction of reflective modular forms.
We prove below that $A_4^\vee(5)$ is also a lattice from this class.
 
By definition, $A_4=\{(a_1,\dots,a_5)\in \ZZ^5: a_1+\dots+a_5=0\}$.
We fix the set of simple roots in $A_4$  
\begin{align*}
&\alpha_1=(1,-1,0,0,0)& &\alpha_2=(0,1,-1,0,0)&\\
&\alpha_3=(0,0,1,-1,0)& &\alpha_4=(0,0,0,1,-1),&
\end{align*}
and  the set of $10$ positive roots in $A_4$
\begin{multline}\label{R-pos} 
R^+(A_4)=\{\alpha_1,\ \alpha_2,\ \alpha_3,\ \alpha_4,\ \alpha_1+\alpha_2,\  
\alpha_2+\alpha_3,\  \alpha_3+\alpha_4,\\
  \alpha_1+\alpha_2+\alpha_3,\  
\alpha_2+\alpha_3+\alpha_4,\  \alpha_1+\alpha_2+\alpha_3+\alpha_4\}.
\end{multline} 
The fundamental weights of $A_4$ are the vectors 
\begin{align*}
&w_1=\left( \frac{4}{5},-\frac{1}{5},-\frac{1}{5},-\frac{1}{5},-\frac{1}{5} 
\right),& &w_2=\left( \frac{3}{5},\frac{3}{5},-\frac{2}{5},-\frac{2}{5},-
\frac{2}{5} \right),&\\
&w_3=\left( \frac{2}{5},\frac{2}{5},\frac{2}{5},-\frac{3}{5},-\frac{3}{5} 
\right),& &w_4=\left( \frac{1}{5},\frac{1}{5},\frac{1}{5},\frac{1}{5},-
\frac{4}{5} \right).&
\end{align*}
So $(\alpha_i,w_j)=\delta_{ij}$ and 
$A_4^\vee/ A_4=\{ 0,w_1,w_2,w_3,w_4\}$. Therefore the renormalization
$$
A_4^\vee(5)=\ZZ w_1+\ZZ w_2+\ZZ w_3+\ZZ w_4,\quad 5(\cdot,\cdot), 
$$
is an even lattice of determinant 125 and its dual lattice is 
$(A_4^\vee(5))^\vee =\frac{1}{5}A_4(5)$.

\begin{lemma}\label{lem2}
The lattice $A_4^\vee(5)$ satisfies the condition $\norm_2$.
\end{lemma}
\begin{proof}
Consider the discriminant group
$D=(A_4^\vee(5))^\vee/A_4^\vee(5)=\frac{1}{5}A_4(5)/A_4^\vee(5)$. We see from above
that $A_4^\vee/A_4$ is the cyclic group of order $5$.
We conclude that  $D\cong \ZZ/5\ZZ \times \ZZ/5\ZZ \times \ZZ/5\ZZ$
is the $5$-group of order $125$.

We see that any two vectors in $\frac{1}{5}A_4(5)$ of norm $\frac{2}{5}$
(or of norm $\frac{4}{5}$) are  not equivalent modulo $A_4^\vee(5)$.
The vectors of norm $\frac{6}{5}$ in 
$\frac{1}{5}A_4(5)$ have $30$ equivalent classes in $D$ and each class 
contains $2$ elements. The vectors of norm $\frac{8}{5}$ in 
$\frac{1}{5}A_4(5)$ have $20$ equivalent classes in $D$ and each class 
contains $3$ elements. The vectors of norm $2$ in $\frac{1}{5}A_4(5)$ have 
$24$ equivalent classes in $D$ and each class contains $5$ elements. 
Moreover,  the above $124$ classes and the zero element form a 
representative system of $D$. We call this representative system the 
\textbf{standard system} of $D$.

As a $5$-group, $D$ is the bouquet of $31$ cyclic subgroups of order $5$. 
In the standard representative system, ten of them are generated by 
$\frac{2}{5}$-vectors and each subgroup contains two $\frac{2}{5}$-vectors 
and two $\frac{8}{5}$-vectors. Fifteen subgroups are generated by 
$\frac{4}{5}$-vectors and each subgroup contains two 
$\frac{4}{5}$-vectors and two $\frac{6}{5}$-vectors.
Six subgroups are generated by $2$-vectors and each subgroup contains four 
$2$-vectors.

For any vector $\beta$ of norm $2$ in $\frac{1}{5}A_4(5)$, the lattice $A_4^
\vee(5)+\ZZ\beta$ is isomorphic to the root lattice $A_4$. 
\end{proof}

From the proof of the above lemma we obtain 

\begin{lemma}\label{lem3}
Let $\Orth(A_4^\vee(5))$ be the integral orthogonal group of $A_4^\vee(5)$ 
and $\Orth(D)$ be the orthogonal group of the discriminant group $D$ of 
$A_4^\vee(5)$. 
\begin{itemize}
\item[(a)] $\Orth(A_4^\vee(5))\cong\Orth(A_4)=W(A_4)\ltimes C_2$, where 
$W(A_4)$ is the Weyl group of $A_4$ and the subgroup $C_2$ is of order $2$ 
and generated by the operator $\mathfrak{z}\mapsto -\mathfrak{z}$.
\item[(b)] In the standard system,  $\Orth(A_4^\vee(5))$ acts transitively 
on the set of classes of the same norm.
\item[(c)] The natural homomorphism $\Orth(A_4^\vee(5))\to \Orth(D)$ is 
surjective.
\end{itemize}
\end{lemma}

We construct a reflective modular form for the lattice
$2U\oplus A_4^\vee(-5)$. To this end, we consider the Kac--Weyl denominator 
function of the affine Lie algebra of type $A_4$ (see \cite{KP}, \cite{G18}) 
\begin{equation}\label{den-func}
\Theta_{A_4}(\tau,\mathfrak{z})=\eta(\tau)^4\prod_{r\in R^+(A_4)}
\frac{\vartheta(\tau, (r,\mathfrak z))}{\eta(\tau)}, \quad \mathfrak{z}\in 
A_4\otimes \CC
\end{equation}
where $R^+(A_4)$ is the set of positive roots of $A_4$ defined in 
\eqref{R-pos}.

It is easy to check that $\Theta_{A_4}$ 
is a weak Jacobi form of weight 2 and 
index 1 for the lattice $A_4^\vee(5)$ (see \cite[Corollary 2.7]{G18} for the case of any 
root system). Moreover, it is anti-invariant under the action of the Weyl 
group $W(A_4)$ and invariant under the action of $C_2$.
From \cite{KP} it follows that  $\Theta_{A_4}(\tau,\mathfrak{z})$ is 
holomorphic at infinity but we do not need this fact here.

In the dual basis
$\mathfrak{z}=z_1w_1+z_2w_2+z_3w_3+z_4w_4\in A_4\otimes \CC$, $z_i\in \CC$,
we have
\begin{equation}\label{den-func-theta}
\begin{split}
\Theta_{A_4}(\tau,\mathfrak{z})=&
\eta^{-6}\vartheta(z_1)\vartheta(z_2)\vartheta(z_3)\vartheta(z_4)\vartheta(
z_1+z_2)\vartheta(z_2+z_3)\vartheta(z_3+z_4)\\
&\vartheta(z_1+z_2+z_3)\vartheta(z_2+z_3+z_4)\vartheta(z_1+z_2+z_3+z_4),
\end{split}
\end{equation}
where $\vartheta(z)=\vartheta(\tau,z)$. 
Next, we define a weak Jacobi form of weight $0$
\begin{equation}
\begin{split}
\Psi_{A_4}(\tau,\mathfrak{z})&=-\frac{(\Theta_{A_4}\lvert T_{-}(2))(\tau,
\mathfrak{z})}{\Theta_{A_4}(\tau,\mathfrak{z})}\\
&=\sum_{\substack{r\in A_4\\ (r,r)=2 }}\exp\left(2\pi i (r,\mathfrak{z}) 
\right) + 4 +O(q) \in J_{0,A_4^\vee(5),1}^w.
\end{split}
\end{equation}
The Fourier expansion above contains all types of singular 
Fourier coefficients of the weak Jacobi form of weight $0$ according to 
Lemma \ref{LemmaN2} and Lemma \ref{lem2}. Therefore, the corresponding Borcherds product
$\Borch(\Psi_{A_4})$ is a holomorphic modular form of weight $2$ 
with respect to $\widetilde{\Orth}^+(2U\oplus A_4^\vee(-5))$. 
Its first Fourier--Jacobi coefficient is equal to 
$\Theta_{A_4}$. Thus, $\Theta_{A_4}$ is a holomorphic Jacobi form of weight 
$2$. Moreover, the modular group of $\Borch(\Psi_{A_4})$  is determined by 
$\Theta_{A_4}$. By Lemma \ref{lem3}, we have
$$
\Phi_{2, A_4^\vee(5)}=\Borch(\Psi_{A_4})\in M_2(\Orth^+(2U\oplus A_4^\vee(-5)),\chi_2),
$$
and by Lemma \ref{LemmaN2} and Lemma \ref{lem3} (b), its divisor is equal to 
\begin{equation}\label{divA4}
\sum_{\substack{ v\in 2U\oplus \frac{1}{5}A_4(-5) \\ (v,v)=-\frac{2}{5}}} 
\cD_v \\
= \Orth^+(2U\oplus A_4^\vee(-5))\cdot \cD_{(0,0,\alpha_1/5,1,0)}.
\end{equation}
The function $\Grit(\Theta_{A_4})$ is a modular form of weight $2$ for $
\Orth^+(2U\oplus A_4^\vee(-5))$ which is anti-invariant under $W(A_4)$ and 
invariant 
under $C_2$. 

\begin{lemma}
Let $\ell$ be a nonzero vector in $L^\vee$. If $\varphi\in J_{k,L,1}$ vanishes to
at least order $m\ge 1$ along the divisor $(\ell, \mathfrak{z})\in \ZZ\tau+\ZZ$, 
then $\Grit(\varphi)$ vanishes to order at least $m$ on $\cD_{(0,0,\ell,1,0)}$.
\end{lemma}
\begin{proof}
In the formula of the action of the parabolic Hecke operator  $T_{-}(m)$ (see \eqref{T(m)}) the action on $\mathfrak{z}$ is linear.
\end{proof}

By the above lemma, we get $\div (\Grit(\Theta_{A_4}))
\supset \div\left(\Borch(\Psi_{A_4})\right)$. Then we have 
$\Grit(\Theta_{A_4})=\Borch(\Psi_{A_4})$ by K\"ocher's 
principle. Thus we have proved the following.

\begin{theorem}\label{tha4}
The following identity is true for the modular form of weight $2$
\begin{equation}
\Phi_{2, A_4^\vee(5)}=\Grit(\Theta_{A_4})=\Borch(\Psi_{A_4})\in 
M_2(\Orth^+(2U\oplus A_4^\vee(-5)),\chi_2),
\end{equation}
and it has reflective divisor \eqref{divA4}. The character $\chi_2$ is of order $2$ and defined by the relations  
$\chi_2|_{\widetilde{\Orth}^+(2U\oplus A_4^\vee(-5))}=1$, $\chi_2|_{C_2}=1$
and  $\chi_2|_{W(A_4)}=\det$.
\end{theorem}

\begin{remark}
The function $\Phi_{2,A_4^\vee(5)}$ is a reflective modular form of singular weight (see \cite{G18} and \cite{GN18} for the definitions of reflective divisors and reflective modular forms) and has been constructed by N.~Scheithauer in another way (see \cite{Sch06}). Scheithauer constructed this function at the zero-dimensional cusp related to $U\oplus U(5)\oplus A_4$ using the lifting from scalar-valued modular forms on congruence subgroups to modular forms for the Weil representation of $\SL_2(\ZZ)$. By \cite[Corollary 1.13.3]{Nik80}, we conclude
\begin{align*}
U\oplus U(5)\oplus A_4&\cong 2U\oplus A_4^\vee(5),
\end{align*}
because the two lattices belong to the same genus. 
Our construction corresponds to the one-dimensional cusp related to the 
decomposition $2U\oplus A_4^\vee(5)$ and it gives the additive Jacobi lifting of this reflective modular form. 
\end{remark}

\begin{remark}{\bf A hyperbolization of the affine Lie algebra 
$\hat{\mathfrak{g}}(A_4)$.}  The result of Theorem \ref{tha4}
has important applications to the theory of Lie algebras. 
It shows that there exists a hyperbolization of the affine 
Lie algebra $\hat{\mathfrak{g}}(A_4)$, i.e. a Lorentzian Kac--Moody algebra
with the following property: the first Fourier--Jacobi coefficient
of the automorphic Kac--Weyl--Borcherds denominator function of such generalized hyperbolic Kac--Moody algebra is the Kac--Weyl denominator function of the affine Lie algebra $\hat{\mathfrak{g}}(A_4)$. 
The generators and relations of this new algebra are defined by the Fourier coefficients of the lift
$\Grit(\Theta_{A_4})$.
This is a new example in a rather short series:
$\hat{\mathfrak{g}}(A_1)$ (see \cite{GN98}),
$\hat{\mathfrak{g}}(4A_1)$ and $\hat{\mathfrak{g}}(3A_2)$
(see \cite{G18}), $\hat{\mathfrak{g}}(A_2)$ (see \cite{GHS13}).
Twenty-three root systems 
of Niemeier lattices are also this type of examples (see \cite{G18}).  
\end{remark}

\section{Applications}\label{sec:4}

\subsection{The proof of Theorem \ref{thmwt2}.}
Now we can prove the main result of the paper that the theta-block
conjecture is true for the theta blocks of type
$\frac{10-\vartheta}{6-\eta}$.
We prove Theorem \ref{thmwt2} by considering a specialisation 
of  the modular form of Theorem \ref{tha4} (compare with the proof of 
\cite[Theorem 8.2]{GPY1}).
It is known that paramodular forms for $\Gamma_t$ can be viewed as modular 
forms for $\widetilde{\Orth}^+(2U\oplus\latt{-2t})$ (see \cite{G94a},
\cite{GN98}).
Let $z\in \CC$, $\mathbf{a}=(a_1,a_2,a_3,a_4)\in \ZZ^4$, and $\mathbf{v}
=a_1w_1+a_2w_2+a_3w_3+a_4w_4\in A_4^\vee(5)$.
The specialisation $\Phi_{2, A_4^\vee(5)}(\tau, z\mathbf{v},\omega)$ 
is a modular form with respect to 
$\widetilde{\Orth}^+(2U\oplus\latt{-2N(\mathbf{a})})$.
In fact, the index $N(\mathbf{a})$ is the half of the (square) norm of the vector $\mathbf{v}$ in $A_4^\vee(5)$.
By \eqref{den-func} and \eqref{den-func-theta} we get 
$\Theta_{A_4}(\tau,z\mathbf{v})=\phi_{2,\mathbf{a}}(\tau, z)$. 
We consider only such $\mathbf{a}$ that $\phi_{2,\mathbf{a}}\not\equiv 0$.
Due to the linear action of the parabolic Hecke operators $T_{-}(m)$
on the coordinate $\mathfrak{z}$, the pull-back 
$\Grit(\Theta_{A_4})(\tau,z\mathbf{v},\omega)$  is equal to 
$\Grit(\phi_{2,\mathbf{a}})\not\equiv 0$. 
The Borcherds product is also described in terms of the operators 
$T_{-}(m)$ (see Theorem \ref{th:Borcherds}). Therefore 
$\Borch(\Psi_{A_4})(\tau, z\mathbf{v},\omega)$ is the Borcherds product
defined by $\Psi_{A_4}(\tau, z\mathbf{v})$. 
Moreover, 
$\Psi_{A_4}(\tau,z\mathbf{v})=-\frac{\phi_{2,\mathbf{a}}|T_{-}(2)}
{\phi_{2,\mathbf{a}}}$ because the parabolic Hecke operator $T_{-}(2)$ 
commutes with the specialisation $\mathfrak{z}=z\mathbf{v}$.
This gives the relation
$$
\Grit(\phi_{2,\mathbf{a}})
=\Borch\left(-\frac{\phi_{2,\mathbf{a}}|T_{-}(2)}
{\phi_{2,\mathbf{a}}}\right).
$$

\subsection{Explicit divisors of the paramodular forms of weight 2
and linear relations between Fourier coefficients.}\label{Subsec:4.2}
Our construction  gives explicit formulas for the divisors of 
the modular forms from Theorem \ref{thmwt2}.

The first modular form corresponds to $\mathbf{a}=(1,1,1,1)$.
In this case  we obtain the first complementary 
Jacobi--Eisenstein series 
$E_{2,25;1}=\eta^{-6}\vartheta^4\vartheta_2^3\vartheta_3^2\vartheta_4\in J_{2,25}$ (see \cite[\S 2]{EZ}).
The singular Fourier coefficients of $\psi_{0,25}=-(E_{2,25;1}\lvert T_{-}(2))/E_{2,25;1}$ are represented by 
$\sing(\psi_{0,25})=\zeta^4+2\zeta^3+3\zeta^2+4\zeta+4$.
Thus the divisor of $\Grit(E_{2,25;1})$ is completely 
defined by divisors of the theta block, i.e. the 
corresponding Borcherds product has no additional divisor.

When the index is larger than $25$, the corresponding Borcherds products have additional divisors in general and the additional divisors will yield certain relations between the Fourier coefficients of theta blocks.  We first recall one example in \cite{GPY1}. 
Let $\mathbf{a}=(1,1,1,2)$, we get $\phi_{2,37}=\eta^{-6}\vartheta^3\vartheta_2^3\vartheta_3^2\vartheta_4\vartheta_5 \in J_{2,37}^{\text{cusp}}$. Note that $\dim J_{2,37}^{\text{cusp}}=1$. Let $\psi_{0,37}=-(\phi_{2,37}\lvert T_{-}(2))/\phi_{2,37}$. The singular Fourier coefficients of $\psi_{0,37}$ are represented by 
$$
\sing(\psi_{0,37})=\zeta^5+\zeta^4+2\zeta^3+3\zeta^2+3\zeta+4+q^6\zeta^{30}.
$$
The coefficient $q^6\zeta^{30}$ determines the divisor which does not 
appear  in the theta block $\phi_{2,37}$. We call it the additional divisor 
of $\Grit(\phi_{2,37})$. This is the last term  in the formula for the 
full divisor of  $\Borch(\psi_{0,37})$: 
\begin{align*}
& 10 \Hum\left(\begin{array}{cc}
0 & 1/2 \\ 
1/2 & 37
\end{array}  \right)+4\Hum\left(\begin{array}{cc}
0 & 1 \\ 
1 & 37
\end{array}  \right)+2\Hum\left(\begin{array}{cc}
0 & 3/2 \\ 
3/2 & 37
\end{array}  \right)\\
&+\Hum\left(\begin{array}{cc}
0 & 2 \\ 
2 & 37
\end{array}  \right)+\Hum\left(\begin{array}{cc}
0 & 5/2 \\ 
5/2 & 37
\end{array}  \right)+
 \Hum\left(\begin{array}{cc}
6 & 15 \\ 
15 & 37
\end{array}  \right).
\end{align*}
The fact that $\Grit(\phi_{2,37})$ vanishes on the additional divisor
is equivalent to the fact (see \cite[Page 170]{GPY1}) that the Fourier coefficients of $\phi_{2,37}$ satisfy the linear relation
\begin{equation}
\forall n,r\in \ZZ, \quad \sum_{a\in \ZZ}c\left(6a^2+na,30a+r;\phi_{2,37}\right)=0.
\end{equation}

Next, we establish similar relations for other Jacobi forms of weight 2 and small indices. We know from \cite{EZ} that the dimensions of the spaces of Jacobi forms of weight 2 and index $43$, $50$, $53$ are all 1. The generators  can be constructed by the theta blocks
\begin{align*}
&\mathbf{a}=(-1,5,-1,-2):& &\phi_{2,43}=\eta^{-6}\vartheta^3\vartheta_2^2\vartheta_3^2\vartheta_4^2\vartheta_5 \in J_{2,43}^{\text{cusp}}\\
&\mathbf{a}=(2,-1,-3,6):& &\phi_{2,50}=\eta^{-6}\vartheta^2\vartheta_2^3\vartheta_3^2\vartheta_4^2\vartheta_6 \in J_{2,50}\\
&\mathbf{a}=(1,-6,3,1):& &\phi_{2,53}=\eta^{-6}\vartheta^3\vartheta_2^2\vartheta_3^2\vartheta_4\vartheta_5\vartheta_6 \in J_{2,53}^{\text{cusp}}.
\end{align*}
We put $\psi_{0,m}=-(\phi_{2,m}\lvert T_{-}(2))/\phi_{2,m}$ for 
$m=43$, $50$, $53$.
Their singular Fourier coefficients are represented by 
\begin{align*}
\sing(\psi_{0,43})&=\zeta^5+2\zeta^4+2\zeta^3+2\zeta^2+3\zeta+4+q^2\zeta^{19}+q^3\zeta^{23},\\
\sing(\psi_{0,50})&=\zeta^6+2\zeta^4+2\zeta^3+3\zeta^2+2\zeta+4+q^5\zeta^{32}+2q^{11}\zeta^{47}+2q^{12}\zeta^{49},\\
\sing(\psi_{0,53})&=\zeta^6+\zeta^5+\zeta^4+2\zeta^3+2\zeta^2+3\zeta+4+q\zeta^{15}+q^{2}\zeta^{21}+q^{6}\zeta^{36}.
\end{align*}
Since the functions $\Grit(\phi_{2,j})$ vanish on the additional divisors, their Fourier coefficients satisfy the similar relations for all additional divisors.
For example, in the case of $\phi_{2,43}$, we obtain two relations
\begin{align*}
&\forall n,r\in \ZZ,& &\sum_{a\in \ZZ}c\left(2a^2+na,19a+r;\phi_{2,43}\right)=0&\\
&\forall n,r\in \ZZ,& &\sum_{a\in \ZZ}c\left(3a^2+na,23a+r;\phi_{2,43}\right)=0.&
\end{align*}

\subsection{Reflective modular form of weight 12}\label{wight12}
We give one more property of the lattice $A_4^\vee(5)$.
\begin{lemma}\label{lem4}
There is a primitive embedding of $A_4^\vee(5)$ into the Leech lattice
$\Lambda_{24}$.  
\end{lemma}

\begin{proof}
In fact, the lattice $A_4^\vee(5)<\Lambda_{24}$ is a fixed point sublattice 
with respect to the automorphism of cycle shape $5^5/1$.  We refer to 
\cite[\S 9]{Sch04} for details.
\end{proof}

Any sublattice of $\Lambda_{24}$ with property $\norm_2$ 
produces a strongly $2$-reflective modular form.
This is the pull-back of the Borcherds reflective modular form 
$\Phi_{12}\in M_{12}(\Orth^+(II_{2,26}), \det)$ 
(see \cite[Theorem 4.2]{GN18}). According to Lemma \ref{lem2}, 
we obtain a new example of reflective modular forms.

\begin{theorem}\label{Phi12}
Consider the primitive sublattice 
$2U\oplus A_4^\vee(-5)\hookrightarrow 
2U\oplus\Lambda_{24}(-1)$
and the embedding of the homogenous domains
$\cD(2U\oplus A^\vee_4(-5))\hookrightarrow \cD(2U\oplus \Lambda_{24}(-1))$.
Then 
$$
\Phi_{12, A_4^\vee(-5)}=
\Phi_{12}|_{\cD(2U\oplus A_4^\vee(-5))}\in
M_{12}(\widetilde\Orth^+(2U\oplus A_4^\vee(-5)), \det)
$$
is a strongly reflective modular form with complete $(-2)$-divisor
$$
\div(\Phi_{12, A_4^\vee(-5)})=\sum_{\substack{r\in 2U\oplus A_4^\vee(-5)\\(r,r)=-2}}\cD_r.
$$
\end{theorem}

\begin{corollary} The modular form $\Phi_{12, A_4^\vee(-5)}$ determines 
a Lorentzian Kac--Moody algebra.  For this algebra, the $2$-reflective Weyl group of $U\oplus A_4^\vee(-5)$ has a Weyl vector of norm $0$, i.e. it has parabolic type.
\end{corollary}
\begin{proof} See a general  construction of Lorentzian Kac--Moody algebras
by $2$-reflective modular forms in  \cite{GN18}.
\end{proof} 

\begin{corollary}
The modular variety 
$\widetilde\Orth^+(2U\oplus A_4^\vee(-5))\setminus 
\cD(2U\oplus A_4^\vee(-5))$ is at least uniruled.
\end{corollary}
\begin{proof}
One can use the automorphic criterion proved 
in \cite{GH}. We note that this modular variety can be considered as a 
moduli space of lattice polarized $K3$ surfaces.
\end{proof}

\subsection{Quasi pull-backs}
We proved the main theorem using a pull-back of the reflective modular form 
$\Phi_{2, A_4^\vee(5)}$. In this subsection we construct more pull-backs 
and quasi pull-backs of $\Phi_{2, A_4^\vee(5)}$ in order to obtain
interesting relations between liftings of multidimensional 
theta blocks. Using the same arguments about quasi pull-backs as in   
\cite{GHS07}--\cite{GN18}, we obtain the next proposition.

\begin{proposition}
Assume that $T=2U\oplus T_0(-1)\hookrightarrow 2U\oplus A_4^\vee(-5)$ is a 
primitive sublattice of signature $(2,n)$ with $n=3$, $4$, $5$. 
We consider the corresponding embedding 
of the homogenous domains 
$\cD(T)\hookrightarrow \cD(2U\oplus A_4^\vee(-5))$
and a finite set
\begin{equation*}
R_{\frac{2}{5}}(T_0^\perp)=
\left\{v\in \frac{1}{5}A_4(-1) : (v,v)=-
\frac{2}{5},\; (v,T_0)=0 \right\}.
\end{equation*}
Let $\cD(T)^\bullet$ be the affine cone of $\cD(T)$. Then the function 
$$ 
\Phi_{2,A_4^\vee(5)}\vert_T=\frac{\Phi_{2,A_4^\vee(5)}(Z)}
{\prod_{v\in R_{\frac{2}{5}}(T_0^\perp)/
{\pm 1}}(Z,v)}\Bigg\vert_{\cD(T)^\bullet}
$$
is a nontrivial modular form of weight 
$2+\frac{1}{2}|R_{\frac{2}{5}}(T_0^\perp)|$ with respect to 
$\widetilde{\Orth}^+(T)$ with a character of order $2$. 
The modular form $\Phi_{2,A_4^\vee(5)}\vert_T$ 
vanishes only on rational quadratic 
divisors of type $\cD_u(T)$, where $u$ is the orthogonal projection of one 
vector $v\in 2U\oplus \frac{1}{5}A_4(-1)$ with $(v,v)=-\frac{2}{5}$ to $T^
\vee$ satisfying $-\frac{2}{5}\leq (u,u)<0$. 
If the set $R_{\frac{2}{5}}(T_0^\perp)$ is non-empty then $F\vert_T$ is a 
cusp form.
\end{proposition}

\begin{example}
Let $T_0=\ZZ \alpha_1+\ZZ \alpha_2+\ZZ \alpha_3$. Then $T_0$ is isomorphic 
to the lattice $A_3(5)$. In this case $T_0^\perp=\ZZ w_4$ and the set 
$R_{\frac{2}{5}}(T_0^\perp)$ is empty.  Let $\mathfrak{z}
_3=z_1\alpha_1+z_2\alpha_2+z_3\alpha_3=(2z_1-z_2)w_1+(2z_2-z_1-
z_3)w_2+(2z_3-z_2)w_3-z_3 w_4$. 

The pull-back on any sublattice of this type
commutes with the additive Jacobi lifting. In the coordinates fixed above we 
have
\begin{equation}\label{jfa3}
\begin{split}
\Theta_{A_4}|_{A_3(5)}(\tau,\mathfrak{z}_3)=&\eta^{-6}
\vartheta(2z_1-z_2)\vartheta(z_1+z_2-z_3)\vartheta(z_1+z_3)
\vartheta(z_1)\\
&\vartheta(2z_2-z_1-z_3)\vartheta(z_2+z_3-z_1)\vartheta(z_2-z_1)\\
&\vartheta(2z_3-z_2)\vartheta(z_3-z_2)\vartheta(z_3) \in J_{2,A_3,5}.
\end{split}
\end{equation}
Therefore, 
$\Grit(\Theta_{A_4}|_{A_3(5)})$ is a modular form of weight $2$ for 
$\widetilde{\Orth}^+(2U\oplus A_3(-5))$ and also a Borcherds product.
\end{example}

\begin{example}
The sublattice $\ZZ\alpha_1+\ZZ\alpha_3$ is isomorphic to $2A_1(5)$. By 
taking $z_2=0$ in (\ref{jfa3}), we get
\begin{equation*}
\Theta_{2A_1(5)}=\frac{\vartheta^2(z_1+z_3)
\vartheta^2(z_1-z_3)\vartheta^2(z_1)\vartheta^2(z_3)
\vartheta(2z_1)\vartheta(2z_3)}{\eta^6}\in J_{2,2A_1,5}.
\end{equation*}
Note that $\dim J_{2,2A_1,5}=1$ and $5$ is the smallest index such that 
there exists a nontrivial Jacobi form of weight $2$ for $2A_1$. Therefore, 
$\Grit(\Theta_{2A_1(5)})$ is a modular form of weight $2$ for 
$\widetilde{\Orth}^+(2U\oplus 2A_1(-5))$ and a Borcherds product.
\end{example}

We note that the Jacobi form $\Theta_{A_4}$ of type 
$10$-$\vartheta/6$-$\eta$ generates a tower of the lifts of quasi pull-backs 
of type $9$-$\vartheta/3$-$\eta$ of weight $3$ and $8$-$\vartheta$ 
of weight $4$.
The theta block conjecture for these theta blocks was proved in 
\cite{GPY1} based on the reflective modular forms constructed 
in \cite{G18}. The functions considered below have fewer parameters than
the modular forms in \cite{G18} but they are cusp forms.
 
\begin{example}
Let $T_0=\ZZ w_1+\ZZ w_2+\ZZ w_3$. Its Gram martix is 
\begin{align*}
&A_0=\left( \begin{array}{ccc}
4 & 3 & 2 \\ 
3 & 6 & 4 \\ 
2 & 4 & 6
\end{array} \right),&  &\det(A_0)=50,& &A_0^{-1}=\frac{1}{5} \left(\begin{array}{ccc}
2 & -1 & 0 \\ 
-1 & 2 & -1 \\ 
0 & -1 & 3/2
\end{array}  \right).
\end{align*}
In this case, we have $T_0^\perp=\ZZ\alpha_4$ and
$
R_{\frac{2}{5}}(T^\perp)=\{\pm\alpha_4/5\}. 
$
We write 
$\mathfrak{z}=
\mathfrak{z}_1+\mathfrak{z}_2$ with $\mathfrak{z}_1\in T_0\otimes\CC$, 
and $\mathfrak{z}_2\in T_0^{\perp}\otimes\CC$.
The quasi pull-back on $T_0\subset A_4^\vee(5)$
can be written in the affine coordinate as the derivative at 
$\mathfrak{z}_2=0$ (see \cite{GHS07} and \cite[\S 8.4]{GHS13}).
In this particular case we have the  differential operator with respect 
to $z_4$ which commutes with the Hecke operators in the Jacobi lifting.
As a result we  obtain the following Jacobi form of weight $3$
of type $9$-$\vartheta/3$-$\eta$
\begin{equation}
\Theta_{T_0}=\frac{\vartheta(z_1)\vartheta(z_2)\vartheta(z_1+z_2)\vartheta^2(z_3)\vartheta^2(z_2+z_3)\vartheta^2(z_1+z_2+z_3)}{\eta^3} \in J_{3,A_0,1}^{\text{cusp}}.
\end{equation}
Therefore, $\Grit(\Theta_{T_0})\in 
S_3(\widetilde{\Orth}^+(2U\oplus A_0(-1)),\chi_2)$ is a {\bf cusp} form of weight $3$ 
with the Borcherds automorphic product constructed by 
$\Psi_{A_4}(\tau,\mathfrak{z})|_{z_4=0}$.
\end{example}

\begin{example} We can continue the construction of quasi pull-back by
setting $z_2=0$. Then we get a Jacobi lifting of {\it canonical} weight with 
Borcherds product in four variables
$$
\Grit(\vartheta^2(z_1)\vartheta^4(z_3)\vartheta^2(z_1+z_3))
 \in S_{4}(\widetilde{\Orth}^+(2U\oplus B_0(-1)),\chi_2),
\quad 
B_0=\begin{pmatrix}4&2\\2&6\end{pmatrix}.
$$
\end{example}

The examples considered above support the following generalization of the 
theta block conjecture formulated in the introduction.

\begin{conjecture}\label{conj2}
Let $\phi\in J_{k,L,1}$. The function $\Grit(\phi)$ has a Borcherds 
product expansion, i.e.
$$
\Grit(\phi)=\Borch\left(-\frac{\phi\lvert T_{-}(2)}{\phi} \right), 
$$
if and only if $\phi$ is a pure theta block of type \eqref{FJtheta} with 
$f(0,\ell)\geq 0$ for all $\ell$ and $\phi$ has vanishing order one in $q$.
\end{conjecture}

\begin{remark}
The ``only if'' part of the above conjecture has an immediate corollary. 
If there exists a non-constant modular form of weight $k$ associated to $
\widetilde{\Orth}^+(2U\oplus L(-1))$ which is simultaneously an additive 
lift and a Borcherds product, then $\rank(L)\leq 8$ and $\rank(L)/2 \leq k 
\leq 12-\rank(L)$. In fact, since $\phi$ has vanishing order one in $q$, the number $A$ in Theorem \ref{th:Borcherds} is equal to $1$. Equation \eqref{FJtheta} defines a Jacobi form for  $L$.
The Fourier coefficients of any Jacobi form of weight $0$ define a generalized $2$-design in the dual lattice $L^\vee$ (see \cite[Proposition 2.6]{G18})
\begin{equation*}\label{eq-2des}
\sum_{\ell\in L^\vee} f(0,\ell)(\ell, \mathfrak{z})^2=
2C(\mathfrak{z},\mathfrak{z})
\qquad \forall\,\mathfrak{z}\in L\otimes \CC.
\end{equation*}
Therefore, the number of $\ell>0$ with non-zero $f(0,\ell)$ is at least $\rank(L)$. In view of the singular weight i.e. $k\geq \frac{1}{2}\rank(L)$ (see e.g. \cite[Corollary 3.2, 3.3]{Bo95}), we prove the above claim.
\end{remark}
\noindent
\textbf{Acknowledgements.} The first author is supported by the Laboratory of Mirror Symmetry NRU HSE (RF government grant, ag. N 14.641.31.0001) and IUF.
The second author is supported by the Labex CEMPI (ANR-11-LABX-0007-01) of the University of Lille. The authors thank the referees for various helpful comments.

\bibliographystyle{amsplain}

\begin{thebibliography}{10}
\bibitem{Bo95}R.E.~Borcherds, {\it Automorphic forms on $O_{s+2,2}(R)$
and infinite products.}
Invent. Math. {\bf 120} (1995), 161--213.

\bibitem{Bo98} R.E.~Borcherds, \textit{Automorphic forms with singularities on Grassmannians.} Invent. Math. \textbf{123} (1998), no. 3, 491--562.

\bibitem{CG} F.~Cl\'ery,   V.~Gritsenko,
\textit{Modular forms of orthogonal type and Jacobi theta-series.}
Abh. Math. Semin. Univ. Hambg. \textbf{83} (2013), 187--217.



\bibitem{EZ} M.~Eichler, D.~Zagier, \textit{The theory of Jacobi
forms.} Progress in Mathematics \textbf{55}. Birkh\"auser, Boston,
Mass., 1985.

\bibitem{G94a} V.~Gritsenko, {\it Modular forms and moduli spaces of
Abelian and $\Kthree$ surfaces.} Algebra i Analiz {\bf 6} (1994),
65--102; English translation in St. Petersburg Math. J. {\bf 6}
(1995), 1179--1208.


\bibitem{G94} V.~Gritsenko, \textit{Irrationality of the moduli spaces of 
polarized abelian surfaces.} Int. Math. Res. Not. IMRN \textbf{6} (1994), 
235--243.


\bibitem{G18} V.~Gritsenko, \textit{Reflective modular forms and their applications.} Uspekhi Mat. Nauk \textbf{73}:5(443) (2018), 53--122. 
English transl. in  Russian Math. Surveys \textbf{73}:5 (2018), 
797--864.

\bibitem {GH} V.~Gritsenko, K.~Hulek, \textit{Uniruledness of orthogonal 
modular varieties.} J. Algebraic Geom. \textbf{23} (2014), 711--725.

\bibitem{GHS07} V.~Gritsenko, K.~Hulek, G.K.~Sankaran,  \textit{The 
Kodaira dimension of the moduli of $\Kthree$ surfaces.} Invent. Math. \textbf{169} 
(2007), 519--567.

\bibitem{GHS09} V. Gritsenko, K. Hulek, G. K. Sankaran,  \textit{Abelianisation of orthogonal groups and the fundamental group of modular varieties.} J. Algebra \textbf{322} (2009), no.2, 463--478.

\bibitem{GHS13} V.~Gritsenko, K.~Hulek, G.K.~Sankaran,  \textit{Moduli of 
$\Kthree$ surfaces and irreducible symplectic manifolds.} Handbook of 
moduli, vol. 1, Adv. Lect. Math. (ALM) vol. 24, 
Intern. Press, Somerville, MA 2013, pp. 459--526.

\bibitem{GN98} V.~Gritsenko, V.V.~Nikulin, 
\textit{Automorphic forms and Lorentzian Kac--Moody algebras. Part II.}  
Internat. J. Math. \textbf{9} (1998), 201--275. 

\bibitem {GN18} V.~Gritsenko, V.V.~Nikulin, \textit{Lorentzian Kac--Moody 
algebras with Weyl groups of $2$-reflections.} Proc. Lond. Math. Soc. (3)
\textbf{116} (2018), no.3, 485--533.

\bibitem{GPY1} V.~Gritsenko, C.~Poor, D.~Yuen,
\textit{Borcherds products everywhere.} J. Number Theory
\textbf{148} (2015), 164--195. 

\bibitem{GSZ} V.~Gritsenko, N. P.~Skoruppa, D.~Zagier, \textit{Theta
blocks.} arXiv:1907.00188.


\bibitem{GW} V.~Gritsenko, H.~Wang,  \textit{Conjecture on theta-blocks of 
order $1$.} Uspekhi Mat. Nauk \textbf{72}:5(437) (2017), 191--192. 
English transl. in  Russian Math. Surveys \textbf{72}:5 (2017), 
968--970.

\bibitem{KP} V.~Kac, D.H.~Peterson,
{\it Infinite-dimensional Lie algebras, theta functions and modular forms.}
Adv. Math. \textbf{53} (1984), 125--264.


\bibitem{Nik80} V.V.~Nikulin, \textit{Integer symmetric bilinear forms and 
some of their geometric applications.} Math. USSR Izv. \textbf{14}, 
103--167 (1980). 

\bibitem{PSY} C.~Poor, J.~Shurman, D.~Yuen, \textit{Theta block Fourier expansions, Borcherds products, and a sequence of Newman and Shanks.} Bull. Aust. Math. Soc. \textbf{98} (2018), 48--59.

\bibitem{Sch04} N.R.~Scheithauer, \textit{Generalized Kac--Moody algebras,
automorphic forms and Conway's group I.} Adv. Math. \textbf{183} (2004), 
240--270.

\bibitem{Sch06} N.R.~Scheithauer, \textit{On the classification of 
automorphic products and generalized Kac--Moody algebras.} Invent. Math. 
\textbf{164} (2006), 641--678.

\end{thebibliography}

\end{document}